\documentclass[12pt]{amsart}
\usepackage[english]{babel}
\usepackage[utf8]{inputenc}
\usepackage{lipsum}
\usepackage{mathrsfs}
\usepackage{stix}
\usepackage{fullpage}
\usepackage{mathtools}
\usepackage{amsmath}
\usepackage{tikz-cd}
\usepackage[colorlinks]{hyperref}
\usepackage{fullpage}
\usepackage{mathtools}
\usepackage{amsmath}
\usepackage{tikz-cd}
\numberwithin{equation}{section}
\theoremstyle{plain}
\newtheorem{thm}{Theorem}[section]
\newtheorem{prop}[thm]{Proposition}
\newtheorem{cor}[thm]{Corollary}

\newtheorem{rem}[thm]{Remark}

\def\R{\mathcal R}
\def\G{\mathbb{G}}
\def\S{\mathfrak{S}}

\begin{document}

 \title[Reverse inequalities  on homogeneous groups]
{Reverse Stein-Weiss,  Hardy-Littlewood-Sobolev, Hardy, Sobolev and Caffarelli-Kohn-Nirenberg inequalities  on homogeneous groups}

\author{Aidyn Kassymov}
\address{
  Aidyn Kassymov:
  \endgraf
   \endgraf
  Department of Mathematics: Analysis, Logic and Discrete Mathematics
  \endgraf
  Ghent University, Belgium
  \endgraf
 	 and
  \endgraf
  Al-Farabi Kazakh National University
  \endgraf
   71 Al-Farabi avenue
   \endgraf
   050040 Almaty
   \endgraf
   Kazakhstan
  \endgraf
	and
\endgraf
  Institute of Mathematics and Mathematical Modeling
  \endgraf
  125 Pushkin str.
  \endgraf
  050010 Almaty
  \endgraf
  Kazakhstan
  \endgraf
  {\it E-mail address} {\rm aidyn.kassymov@ugent.be} and {\rm kassymov@math.kz}}

  \author[M. Ruzhansky]{Michael Ruzhansky}
\address{
	Michael Ruzhansky:
	 \endgraf
  Department of Mathematics: Analysis, Logic and Discrete Mathematics
  \endgraf
  Ghent University, Belgium
  \endgraf
  and
  \endgraf
  School of Mathematical Sciences
    \endgraf
    Queen Mary University of London
  \endgraf
  United Kingdom
	\endgraf
  {\it E-mail address} {\rm michael.ruzhansky@ugent.be}
}

\author[D. Suragan]{Durvudkhan Suragan}
\address{
	Durvudkhan Suragan:
	\endgraf
	Department of Mathematics, Nazarbayev University
	\endgraf
	53 Kabanbay Batyr Ave, Nur-Sultan 010000
	\endgraf
	Kazakhstan
	\endgraf
	{\it E-mail address} {\rm durvudkhan.suragan@nu.edu.kz}}

\thanks{
The authors were supported in parts by the FWO Odysseus Project, the Leverhulme Grant RPG-2017-151 and by EPSRC Grant EP/R003025/1, as well as NU CRG 091019CRP2120 and NU FDCRG 240919FD3901.
\
}


     \begin{abstract}
In this note we prove the reverse Stein-Weiss inequality on general homogeneous Lie groups. The obtained results extend previously known inequalities. Special properties of homogeneous norms and the reverse integral Hardy inequality play  key roles in our proofs. Also, we show reverse Hardy, Hardy-Littlewood-Sobolev, $L^{p}$-Sobolev and $L^{p}$-Caffarelli-Kohn-Nirenberg inequalities on homogeneous groups. 
     \end{abstract}
     \maketitle

\section{Introduction}

In one of their pioneering work \cite{HL28}, Hardy and Littlewood considered the one dimensional fractional integral operator on $(0,\infty)$ given by
\begin{equation}\label{1Doper}
T_{\lambda}u(x)=\int_{0}^{\infty}\frac{u(y)}{|x-y|^{\lambda}}dy,\,\,\,\,0<\lambda<1,
\end{equation}
and proved the following theorem:
\begin{thm}\label{1DHLS28}
Let $1<p<q<\infty$ and $u\in L^{p}(0,\infty)$  with $\frac{1}{q}=\frac{1}{p}+\lambda-1$. Then
\begin{equation}
\|T_{\lambda}u\|_{L^{q}(0,\infty)}\leq C \|u\|_{L^{p}(0,\infty)},
\end{equation}
where $C$ is a positive constant independent of $u$.
\end{thm}
The $N$-dimensional analogue of \eqref{1Doper} can be written by the formula:
\begin{equation}\label{NDoper}
I_{\lambda}u(x)=\int_{\mathbb{R}^{N}}\frac{u(y)}{|x-y|^{\lambda}}dy,\,\,\,\,0<\lambda<N.
\end{equation}
The $N$-dimensional case of  Theorem \ref{1DHLS28} was extended by Sobolev in
\cite{Sob38}:
\begin{thm}\label{THM:HLS}
Let $1<p<q<\infty$, $u\in L^{p}(\mathbb{R}^{N})$  with $\frac{1}{q}=\frac{1}{p}+\frac{\lambda}{N}-1$. Then
\begin{equation}
\|I_{\lambda}u\|_{L^{q}(\mathbb{R}^{N})}\leq C \|u\|_{L^{p}(\mathbb{R}^{N})},
\end{equation}
where $C$ is a positive constant independent of $u$.
\end{thm}
Later, in \cite{StWe58} Stein and Weiss obtained the following two-weight extention of the  Hardy-Littlewood-Sobolev inequality, which is known as the Stein-Weiss inequality or weighted Hardy-Littlewood-Sobolev inequality.
\begin{thm}\label{Classiacal_Stein-Weiss_inequality}
Let $0<\lambda<N$, $1<p<\infty$, $\alpha<\frac{N(p-1)}{p}$, $\beta<\frac{N}{q}$, $\alpha+\beta\geq0$ and $\frac{1}{q}=\frac{1}{p}+\frac{\lambda+\alpha+\beta}{N}-1$. If $1<p\leq q<\infty$, then
\begin{equation}
\||x|^{-\beta}I_{\lambda}u\|_{L^{q}(\mathbb{R}^{N})}\leq C \||x|^{\alpha}u\|_{L^{p}(\mathbb{R}^{N})},
\end{equation}
where $C$ is a positive constant independent of $u$.
\end{thm}

The Hardy-Littlewood-Sobolev inequality on Euclidean spaces and the regularity of fractional integrals was studied in \cite{CF}, \cite{FM}, \cite{MW}  and \cite{Per}. On the Heisenberg group, Folland and Stein in \cite{FS74} obtained the Hardy-Littlewood-Sobolev inequality and in \cite{HLZ} the authors also proved an analogue of the Stein-Weiss inequality.
In  \cite{GMS} the authors studied the Stein-Weiss inequality on the Carnot groups. On homogeneous Lie groups, the Hardy-Littlewood-Sobolev and Stein-Weiss inequalities were obtained in \cite{RY} and \cite{KRS}.
In \cite{JD} the author proved the Stein-Weiss inequality on the Euclidean half-space.

The reverse Hardy-Littlewood-Sobolev inequality in the Euclidean space was obtained in the works  \cite{JZ}, \cite{CDDFF} and \cite{NN}. In \cite{CLT1} and \cite{CLT}, the authors obtained the reverse Stein-Weiss inequality on the Euclidean space and half-space, respectively. In this paper, we first show the reverse Stein-Weiss inequality on the homogeneous  groups.
In the proof we use special properties of  homogeneous norms of the homogeneous Lie groups and reverse integral Hardy inequality, which are playing  key roles in our calculations. Thus, in Theorem \ref{stein-weiss3} we establish the reverse Stein-Weiss inequality on general homogeneous groups based on the reverse integral Hardy inequalities with one negative exponent.
In particular, the obtained result recovers the previously  known results of Abelian (Euclidean), Heisenberg, Carnot groups since the class of the homogeneous Lie groups contains those and since we can work with an arbitrary homogeneous quasi-norm.
Note that in this direction systematic studies of different functional inequalities on  general homogeneous (Lie) groups were initiated by the paper \cite{RSAM}. We refer to this and other papers by the authors (e.g. \cite{RSY1}) for further discussions.

We also note that the best constant in the Hardy-Littlewood-Sobolev inequality on the Heisenberg group is now known, see Frank and Lieb \cite{FL12} (in the Euclidean case this was done earlier by Lieb in \cite{Lie83}). The expression for the best constant depends on the particular quasi-norm used and may change for a different choice of a quasi-norm. 

A multidimensional version of one of the well-known inequalities of G.H. Hardy is
\begin{equation}\label{dirhar}
\left\|\frac{f}{|x|}\right\|_{L^{p}(\mathbb{R}^{N})}\leq \frac{p}{N-p}\|\nabla f\|_{L^{p}(\mathbb{R}^{N})},\,\,\,\,1<p<N,
\end{equation}
where $f\in C^{\infty}_{0}(\mathbb{R}^{N})$, $\nabla$ is the Euclidean gradient and constant $\frac{p}{N-p}$ is sharp. The Hardy inequality was intensively studied. For example, the Hardy inequalities were considered on the Euclidean space in \cite{HHAT}, \cite{HL}, on the Heisenberg group in \cite{DA}, \cite{DGP}, on stratified groups in \cite{CCR}, \cite{RSS1, RSS2, RSS3}, \cite{RSAMS}, \cite{RY1}, for the vector fields in \cite{RSS4} and \cite{RSVF}, on homogeneous groups in \cite{RSAM}, \cite{RSY1}.
In \cite{RSAM} (see e.g., \cite{RS_book}), authors showed the Hardy inequality with radial derivative on homogeneous groups in the following form:
\begin{thm}\label{dirHardy}
Let $\G$ be a homogeneous group of homogeneous dimension $Q$. Let $|\cdot|$
be a homogeneous quasi-norm on $\G$. Let $1 < p < Q$.
Let $f\in C^{\infty}_{0}(\G\setminus\{0\})$ be a complex-valued function. Then
\begin{equation}\label{dirharhom}
\left\|\frac{f}{|x|}\right\|_{L^{p}(\G)}\leq \frac{p}{Q-p}\|\R f\|_{L^{p}(\G)},\,\,\,\,\, 1 < p < Q,
\end{equation}
where $\R=\frac{d}{d|x|}$ is the radial derivative. The constant $\frac{p}{Q-p}$ is sharp.
\end{thm}
In Abelian case $(\mathbb{R}^{N},+)$ with $Q=N$ and $|\cdot|=|\cdot|_{E}$ where $|\cdot|_{E}$ is the standard Euclidean distance, from \eqref{dirharhom} we have
\begin{equation}\label{dirharr}
\left\|\frac{f}{|x|_{E}}\right\|_{L^{p}(\mathbb{R}^{N})}\leq \frac{p}{N-p}\left\|\frac{x}{|x|_{E}}\cdot \nabla f\right\|_{L^{p}(\mathbb{R}^{N})},\,\,\,\,\, 1 < p < N,
\end{equation}
where $\nabla$ is the standard Euclidean gradient. By using the Cauchy-Schwartz inequality from the inequality \eqref{dirharr}, we get \eqref{dirhar}.
 In this note, we show the reverse Hardy inequality on homogeneous general Lie groups.

 We are also interested in Sobolev and Caffarelli-Kohn-Nirenberg inequalities, let us recall them briefly. In the classical work of Sobolev, he showed the following inequality:
\begin{equation}
\|u\|_{L^{p^{*}}(\mathbb{R}^{N})}\leq C\|\nabla u\|_{L^{p}(\mathbb{R}^{N})},\,\,\,1<p<N,
\end{equation}
where $p^{*}=\frac{Np}{N-p}$. The Sobolev inequalities on stratified groups were obtained in \cite{RSY1, RSY2, RSY3}. In \cite{RSY3}, on  homogeneous groups, the authors showed $L^{p}$-Sobolev inequality in the following form:
\begin{thm}\label{dirSob}
Let $\G$ be a homogeneous group of homogeneous dimension $Q$. Let $|\cdot|$
be a homogeneous quasi-norm on $\G$. Let $1 < p < Q$.
Let $f\in C^{\infty}_{0}(\G\setminus\{0\})$ be a complex-valued function. Then
\begin{equation}\label{dirsob}
\left\|f\right\|_{L^{p}(\G)}\leq \frac{p}{Q}\|\mathbb{E} f\|_{L^{p}(\G)},\,\,\,\,\, 1 < p < \infty,
\end{equation}
where $\mathbb{E}=|x|\frac{d}{d|x|}$ is the Euler operator. The constant $\frac{p}{Q}$ is sharp.
\end{thm}
In Abelian case $(\mathbb{R}^{N},+)$ with $Q=N$ and $|\cdot|=|\cdot|_{E}$ where $|\cdot|_{E}$ is the standard Euclidean distance, from \eqref{dirharhom} we have
\begin{equation}\label{dirsobr}
\left\|f\right\|_{L^{p}(\mathbb{R}^{N})}\leq \frac{p}{N-p}\left\|x\cdot \nabla f\right\|_{L^{p}(\mathbb{R}^{N})},\,\,\,\,\, 1 < p < \infty,
\end{equation}
where $\nabla$ is the standard Euclidean gradient. In the Euclidean case, this $L^{p}$-Sobolev  inequality  has been considered in \cite{OS}.

In their fundamental work \cite{CKN}, L. Caffarelli, R. Kohn and L. Nirenberg established:
\begin{thm}
Let $N\geq1$, and let $l_1$, $l_2$, $l_3$, $a, \, b, \, d,\, \delta \in \mathbb{R}$ be such that $l_1, l_2 \geq 1$,
$l_3 > 0, \,\,0 \leq \delta \leq 1,$ and
\begin{equation}
\frac{1}{l_1}+\frac{a}{N},\,\,\,\frac{1}{l_2}+\frac{b}{N},\,\,\,\frac{1}{l_3}+\frac{\delta d+(1-\delta) b}{N}>0.
\end{equation}
Then,
\begin{equation}
\||x|^{\delta d+(1-\delta) b}u\|_{L^{l_{3}}(\mathbb{R}^{N})}\leq C\||x|^{a}\nabla u\|^{\delta}_{L^{l_{1}}(\mathbb{R}^{N})}\||x|^{b} u\|^{1-\delta}_{L^{l_{2}}(\mathbb{R}^{N})},\,\,\,u\in C^{\infty}_{c}(\mathbb{R}^{N}),
\end{equation}
if and only if
\begin{multline}
\frac{1}{l_3}+\frac{\delta d+(1-\delta) b}{N}=\delta\left(\frac{1}{l_{1}}+\frac{a-1}{N}\right)+
(1-\delta)\left(\frac{1}{l_2}+\frac{b}{N}\right),\\
a-d\geq0,\,\,\,\,\text{if}\,\,\,\delta>0,\\
a-d\leq1,\,\,\,\,\text{if}\,\,\,\delta>0\,\,\,\text{and}\,\,\,\frac{1}{l_3}+\frac{\delta d+(1-\delta) b}{N}=\frac{1}{l_1}+\frac{a-1}{N},
\end{multline}
where $C$ is a positive constant independent of $u$.
\end{thm}
On homogeneous groups,  $L^{p}$-Caffarelli-Kohn-Nirenberg inequality was obtained in \cite{ORS}. 
\begin{thm}\label{LpCKN}
Let $\mathbb{G}$ be a homogeneous group
of homogeneous dimension $Q$ and let $\alpha,\,\beta\in \mathbb{R}$.
Then for all complex-valued functions $f\in C^{\infty}_{0}(\mathbb{G}\backslash\{0\}),$ $1<p<\infty,$
and any homogeneous quasi-norm $|\cdot|$ on $\mathbb{G}$ we have
\begin{equation}\label{eqLpCKN}
\frac{|Q-\gamma|}{p}
\left\|\frac{f}{|x|^{\frac{\gamma}{p}}}\right\|^{p}_{L^{p}(\mathbb{G})}\leq
\left\|\frac{1}{|x|^{\alpha}}\mathcal{R} f\right\|_{L^{p}(\mathbb{G})}\left\|\frac{f}{|x|^{\frac{\beta}{p-1}}}\right\|^{p-1}_{L^{p}(\mathbb{G})},
\end{equation}
where $\gamma=\alpha+\beta+1$. If $\gamma\neq Q$ then the constant $\frac{|Q-\gamma|}{p}$ is sharp.
\end{thm}
The Caffarelli-Kohn-Nirenberg inequalities, on stratified and homogeneous groups  were obtained in \cite{RSY1, RSY2, RSY3}. On homogeneous groups in \cite{ORS},  a general  result which was obtained in particular cases gives the Hardy inequality \eqref{dirharhom} and $L^{p}$-Sobolev inequality \eqref{dirsob}.

 In this note, we show the reverse Hardy, Hardy-Littlewood-Sobolev, $L^{p}$-Sobolev, $L^{p}$-Caffarelli-Kohn-Nirenberg inequalities on homogeneous Lie groups. The main idea for proving such reverse inequalities was developed in \cite{RSAM}. In addition, to the best of our knowledge, such reverse
 inequalities above on a  homogeneous group $\G$ are new even in the Euclidean
case.

Summarising our main results of the present note, we prove the following facts:
\begin{itemize}
\item The reverse Stein-Weiss and Hardy-Littlewood-Sobolev inequalities on  homogeneous groups;
\item The reverse Hardy inequality  on  homogeneous groups;
\item The reverse $L^{p}$-Sobolev inequality  on homogeneous groups;
\item The reverse $L^{p}$-Caffarelli-Kohn-Nirenberg inequality on homogeneous groups.
\end{itemize}

\section{Main results}
\label{SEC:2}
 Let us recall that a Lie group (on $\mathbb{R}^{N}$) $\mathbb{G}$ with the dilation
$$D_{\lambda}(x):=(\lambda^{\nu_{1}}x_{1},\ldots,\lambda^{\nu_{N}}x_{N}),\; \nu_{1},\ldots, \nu_{n}>0,\; D_{\lambda}:\mathbb{R}^{N}\rightarrow\mathbb{R}^{N},$$
which is an automorphism of the group $\mathbb{G}$ for each $\lambda>0,$
is called a {\em homogeneous (Lie) group}. For simplicity, throughout this paper we use the notation $\lambda x$ for the dilation $D_{\lambda}.$  The homogeneous dimension of the homogeneous group $\mathbb{G}$ is denoted by $Q:=\nu_{1}+\ldots+\nu_{N}.$
Also, in this note we denote a homogeneous quasi-norm on $\mathbb{G}$ by $|x|$, which
is a continuous non-negative function
\begin{equation}
\mathbb{G}\ni x\mapsto |x|\in[0,\infty),
\end{equation}
with the properties

\begin{itemize}
	\item[i)] $|x|=|x^{-1}|$ for all $x\in\mathbb{G}$,
	\item[ii)] $|\lambda x|=\lambda |x|$ for all $x\in \mathbb{G}$ and $\lambda>0$,
	\item[iii)] $|x|=0$ iff $x=0$.
\end{itemize}
Moreover, the following polarisation formula on homogeneous Lie groups will be used in our proofs:
there is a (unique)
positive Borel measure $\sigma$ on the
unit quasi-sphere
$
\mathfrak{S}:=\{x\in \mathbb{G}:\,|x|=1\},
$
so that for every $f\in L^{1}(\mathbb{G})$ we have
\begin{equation}\label{EQ:polar}
\int_{\mathbb{G}}f(x)dx=\int_{0}^{\infty}
\int_{\mathfrak{S}}f(ry)r^{Q-1}d\sigma(y)dr.
\end{equation}
The quasi-ball centred at $x \in \mathbb{G}$ with radius $R > 0$ can be defined by
\begin{equation*}
B(x,R) := \{y \in\mathbb {G} : |x^{-1} y|< R\}.
\end{equation*}
We refer to \cite{FS1} for the original appearance of such groups, and to \cite{FR,RS_book} for a recent comprehensive treatment.

Let us consider the integral operator
\begin{equation}
I_{\lambda,|\cdot|}u(x)=|x|^{\lambda}*u=\int_{\mathbb{G}}|y^{-1} x|^{\lambda}u(y)dy,\,\,\,\lambda>0,
\end{equation}
where $*$ is the convolution.
Let us recall briefly the reverse H\"{o}lder's inequality.
\begin{thm}[\cite{AF}, Theorem 2.12, p. 27)]\label{Hol}
Let $p\in(0,1)$, so that $p'=\frac{p}{p-1}<0$. If non-negative functions satisfy $f\in L^{p}(\G)$ and $0<\int_{\G}g^{p'}(x)dx<+\infty,$ we have
\begin{equation}\label{Holin}
\int_{\G}f(x)g(x)dx\geq\left(\int_{\G}f^{p}(x)dx\right)^{\frac{1}{p}}\left(\int_{\G}g^{p'}(x)dx\right)^{\frac{1}{p'}}.
\end{equation}
\end{thm}
Let us also recall a well-known fact about quasi-norms.
\begin{prop}[\cite{FR}, Theorem 3.1.39] \label{prop_quasi_norm}
Let $\mathbb{G}$ be a homogeneous Lie group. Then there exists a homogeneous
quasi-norm on $\mathbb{G}$ which is a norm, that is, a homogeneous quasi-norm $|\cdot|$
which satisfies the triangle inequality
\begin{equation}\label{tri}
|x y|\leq |x| + |y|, \,\,\,\forall x, y \in \mathbb{G}.
\end{equation}
Furthermore, all homogeneous quasi-norms on $\mathbb{G}$ are equivalent.
\end{prop}

The next theorem is the reverse integral version of Hardy inequalities on general homogeneous groups that will be instrumental in our proof.

\begin{thm}[\cite{KRS2}]\label{integral_hardy}
Let $\mathbb{G}$ be a homogeneous
group of homogeneous dimension $Q$. Assume that $p\in(0,1)$ and $q<0$.  Suppose that $W,U\geq0$ are locally integrable functions on $\mathbb{G}$. Then the inequality
\begin{equation}\label{5.2}
\left[\int_{\mathbb{G}}\left(\int_{B(0,|x|)}f(y)dy\right)^{q}W(x)dx\right]^{\frac{1}{q}}\geq C_{1}(p,q)\left(\int_{\mathbb{G}}f^{p}(x)U(x)dx\right)^{\frac{1}{p}}
\end{equation}
holds for some $C_{1}(p,q)>0$ and all non-negative measurable functions $f,$ if and only if
\begin{equation}\label{5.2.1}
0< A_{1}:=\inf_{x\neq a}\left[\left(\int_{\mathbb{G}\setminus B(0,|x|)}W(y)dy\right)^{\frac{1}{q}}\left(\int_{B(0,|x|)}U^{1-p'}(y)dy\right)^{\frac{1}{p'}}\right].
\end{equation}
Moreover, the biggest constant $C_{1}(p,q)$ in \eqref{5.2} has the following relation to $A_{1}$:
\begin{equation}\label{C}
A_{1}\geq C_{1}(p,q)\geq \left(\frac{p'}{p'+q}\right)^{-\frac{1}{q}}\left(\frac{q}{p'+q}\right)^{-\frac{1}{p'}}A_{1}.
\end{equation}
Also, inequality
\begin{equation}\label{5.4}
\left[\int_{\mathbb{G}}\left(\int_{\G\setminus B(0,|x|)}f(y)dy\right)^{q}W(x)dx\right]^{\frac{1}{q}}\geq C_{2}(p,q)\left(\int_{\mathbb{G}}f^{p}(x)U(x)dx\right)^{\frac{1}{p}}
\end{equation}
holds for some $C_{2}(p,q)>0$ and all non-negative measurable functions $f,$ if and only if
\begin{equation}\label{5.4.1}
0< A_{2}:=\inf_{x\neq a}\left[\left(\int_{ B(0,|x|)}W(y)dy\right)^{\frac{1}{q}}\left(\int_{\mathbb{G}\setminus B(0,|x|)}U^{1-p'}(y)dy\right)^{\frac{1}{p'}}\right].
\end{equation}
Moreover, the biggest constant $C_{2}(p,q)$ in \eqref{5.4} has the following relation to $A_{2}$:
\begin{equation}\label{C}
A_{2}\geq C_{2}(p,q)\geq \left(\frac{p'}{p'+q}\right)^{-\frac{1}{q}}\left(\frac{q}{p'+q}\right)^{-\frac{1}{p'}}A_{2}.
\end{equation}

\end{thm}
\begin{rem}
In our sense, the negative exponent $q<0$ of  $0$, we understand in the following form:
\begin{equation*}
0^{q}=(+\infty)^{-q}=+\infty,\,\,\,\,\text{and}\,\,\,\,\,\,0^{-q}=(+\infty)^{q}=0.
\end{equation*}
\end{rem}
Now we formulate the reverse Stein-Weiss inequality on $\mathbb{G}$.
\begin{thm}\label{stein-weiss3}
Let $\mathbb{G}$ be a homogeneous group of homogeneous dimension $Q\geq1$ and let $|\cdot|$ be an arbitrary homogeneous quasi-norm on $\mathbb{G}$.
Let $\lambda>0$, $p,q'\in(0,1)$, $0\leq\alpha<-\frac{Q}{q}$, $0\leq\beta<-\frac{Q}{p'}$,  $\frac{1}{q'}+\frac{1}{p}=\frac{\alpha+\beta+\lambda}{Q}+2$, where $\frac{1}{p}+\frac{1}{p'}=1$ and $\frac{1}{q}+\frac{1}{q'}=1$. Then for all non-negative functions $f\in L^{q'}(\mathbb{G})$ and $h\in L^{p}(\mathbb{G})$ we have
\begin{equation}\label{stein-weiss}
\int_{\mathbb{G}}\int_{\mathbb{G}}|x|^{\alpha}|y^{-1}x|^{\lambda}f(x)h(y)|y|^{\beta}dxdy\geq C\|f\|_{L^{q'}(\mathbb{G})}\|h\|_{L^{p}(\mathbb{G})},
\end{equation}
where $C$ is a positive constant independent of $f$ and $h$.

\end{thm}
\begin{cor}
	By setting $\alpha=\beta=0$ we get the reverse Hardy-Littlewood-Sobolev inequality on the homogeneous groups, in the following form:
\begin{equation}\label{hls}
\int_{\mathbb{G}}\int_{\mathbb{G}}|y^{-1}x|^{\lambda}f(x)h(y)dxdy\geq C\|f\|_{L^{q'}(\mathbb{G})}\|h\|_{L^{p}(\mathbb{G})},
\end{equation}
for all non-negative functions $f\in L^{q'}(\mathbb{G})$ and $h\in L^{p}(\mathbb{G})$ with $\lambda>0$, $p,q'\in(0,1)$, $\frac{1}{q'}+\frac{1}{p}=\frac{\lambda}{Q}+2$, where $\frac{1}{p}+\frac{1}{p'}=1$ and $\frac{1}{q}+\frac{1}{q'}=1$. 
\end{cor}
\begin{rem}
In the Abelian (Euclidean) case ${\mathbb G}=(\mathbb R^{N},+)$, hence $Q=N$ and $|\cdot|$ can be any homogeneous quasi-norm  on $\mathbb R^{N}$, in particular with the usual Euclidean distance, i.e. $|\cdot|=\|\cdot\|_{E}$, this was investigated in \cite{CLT}.
\end{rem}
\begin{proof}[Proof of Theorem \ref{stein-weiss3}]
By using reverse H\"{o}lder's inequality with $\frac{1}{q}+\frac{1}{q'}=1$ (Theorem \ref{Hol}) in \eqref{stein-weiss}, we calculate,
\begin{equation*}
\begin{split}
\int_{\G}\int_{\G}|x|^{\alpha}f(x)|y^{-1}x|^{\lambda}&h(y)|y|^{\beta}dydx=\int_{\G}\left(\int_{\G}|x|^{\alpha}|y^{-1}x|^{\lambda}h(y)|y|^{\beta}dy\right)f(x)dx\\&
\stackrel{\eqref{Holin}}\geq\left(\int_{\mathbb{G}}\left(\int_{\G}|x|^{\alpha}|y^{-1}x|^{\lambda}h(y)|y|^{\beta}dy\right)^{q}dx\right)^{\frac{1}{q}}\|f\|_{L^{q'}(\G)}.
\end{split}
\end{equation*}
So for \eqref{stein-weiss}, it is enough to show that
\begin{equation*}
\left(\int_{\mathbb{G}}\left(\int_{\G}|x|^{\alpha}|y^{-1}x|^{\lambda}h(y)|y|^{\beta}dy\right)^{q}dx\right)^{\frac{1}{q}}\geq C\|h\|_{L^{p}(\mathbb{G})},
\end{equation*}
and by changing $u(y)=h(y)|y|^{\beta}$, this is equivalent to
\begin{equation*}
\int_{\mathbb{G}}\left(\int_{\mathbb{G}}|x|^{\alpha}|y^{-1} x|^{\lambda}u(y)dy\right)^{q}dx\leq C\||y|^{-\beta}u\|^{q}_{L^{p}(\mathbb{G})}.
\end{equation*}
We have that
\begin{equation*}
\int_{\mathbb{G}}|x|^{\alpha}|y^{-1} x|^{\lambda}u(y)dy\geq\int_{B\left(0,\frac{|x|}{2}\right)}|x|^{\alpha}|y^{-1} x|^{\lambda}u(y)dy,
\end{equation*}
then
\begin{equation*}
\left(\int_{\mathbb{G}}|x|^{\alpha}|y^{-1} x|^{\lambda}u(y)dy\right)^{q}\stackrel{q<0}\leq\left(\int_{B\left(0,\frac{|x|}{2}\right)}|x|^{\alpha}|y^{-1} x|^{\lambda}u(y)dy\right)^{q}.
\end{equation*}
Therefore, we obtain
\begin{multline}\label{i1}
\left(\int_{\mathbb{G}}|x|^{\alpha q}\left(\int_{\mathbb{G}}|y^{-1} x|^{\lambda}u(y)dy\right)^{q}dx\right)^{\frac{1}{q}}\\
\stackrel{q<0}\geq\left(\int_{\mathbb{G}}|x|^{\alpha q}\left(\int_{B\left(0,\frac{|x|}{2}\right)}|y^{-1} x|^{\lambda}u(y)dy\right)^{q}dx\right)^{\frac{1}{q}}:=I^{\frac{1}{q}}_{1}.
\end{multline}
Similarly with \eqref{i1}, we have

\begin{multline}\label{i3}
\left(\int_{\mathbb{G}}|x|^{\alpha q}\left(\int_{\mathbb{G}}|y^{-1} x|^{\lambda}u(y)dy\right)^{q}dx\right)^{\frac{1}{q}}\\
\stackrel{q<0}\geq\left(\int_{\mathbb{G}}|x|^{\alpha q}\left(\int_{\mathbb{\G}\setminus B(0,2|x|)}|y^{-1} x|^{\lambda}u(y)dy\right)^{q}dx\right)^{\frac{1}{q}}:=I^{\frac{1}{q}}_{2}.
\end{multline}
By summarising above facts, from \eqref{i1}-\eqref{i3}, we have
\begin{equation}\label{i13}
\left(\int_{\mathbb{G}}|x|^{\alpha q}\left(\int_{\mathbb{G}}|y^{-1} x|^{\lambda}u(y)dy\right)^{q}dx\right)^{\frac{1}{q}}\geq \frac{I^{\frac{1}{q}}_{1}}{2}+\frac{I^{\frac{1}{q}}_{2}}{2}.
\end{equation}

From now on, in view of Proposition \ref{prop_quasi_norm} we can assume that our quasi-norm is actually a norm.

\textbf{Step 1.} Let us consider $I_{1}$.   By using Proposition \ref{prop_quasi_norm} and the properties of the quasi-norm with  $|y|\leq\frac{|x|}{2}$, we get
\begin{equation}
|x|=|x^{-1}|=|x^{-1}y y^{-1}|\stackrel{\eqref{tri}}\leq |x^{-1} y|+|y^{-1}|=|y^{-1} x|+|y|\leq |y^{-1} x|+\frac{|x|}{2}.
\end{equation}
Then for any $\lambda>0$, we have
$$2^{-\lambda}|x|^{\lambda}\leq |y^{-1} x|^{\lambda}.$$
It means,
$$2^{-\lambda}\int_{B\left(0,\frac{|x|}{2}\right)}|x|^{\lambda}u(y)dy\leq\int_{B\left(0,\frac{|x|}{2}\right)}|y^{-1}x|^{\lambda}u(y)dy,$$
so that
$$\left(\int_{B\left(0,\frac{|x|}{2}\right)}|y^{-1}x|^{\lambda}u(y)dy\right)^{q}\leq 2^{-\lambda q}\left(\int_{B\left(0,\frac{|x|}{2}\right)}|x|^{\lambda}u(y)dy\right)^{q}.$$
Therefore, we get
\begin{equation*}
\begin{split}
I_{1}&=\int_{\mathbb{G}}|x|^{\alpha q}\left(\int_{B\left(0,\frac{|x|}{2}\right)}|y^{-1} x|^{\lambda}u(y)dy\right)^{q}dx\\&
\leq 2^{-\lambda q}\int_{\mathbb{G}}|x|^{(\alpha+\lambda)q}\left(\int_{B\left(0,\frac{|x|}{2}\right)}u(y)dy\right)^{q}dx.
\end{split}
\end{equation*}
If condition \eqref{5.2.1} in Theorem \ref{integral_hardy} with $W(x)=|x|^{(\alpha+\lambda)q}$ and $U(y)=|y|^{-\beta p}$ in \eqref{5.2} is satisfied, then we have
\begin{equation*}
I_{1}\leq2^{-\lambda q}\int_{\mathbb{G}}\left(\int_{B(0,\frac{|x|}{2})}u(y)dy\right)^{q}|x|^{(\alpha+\lambda)q}dx
\leq C_{1}\||y|^{-\beta}u\|^{q}_{L^{p}(\mathbb{G})}.
\end{equation*}
Let us verify condition \eqref{5.2.1}. By using assumption $\beta<-\frac{Q}{p'}$, we obtain
$$\frac{1}{p}+\frac{1}{q'}=\frac{\alpha+\beta+\lambda}{Q}+2<\frac{\alpha+\lambda}{Q}-\frac{1}{p'}+2,$$
that is, $\frac{Q+(\alpha+\lambda)q}{Qq}>0,$ then $Q+(\alpha+\lambda)q<0$ and by  using the polar decomposition \eqref{EQ:polar}:
\begin{equation*}
\begin{split}
\left(\int_{\mathbb{G}\setminus B(0,|x|)}W(y)dy\right)^{\frac{1}{q}}&=
\left(\int_{\mathbb{G}\setminus B(0,|x|)}|y|^{(\alpha+\lambda)q}dy\right)^{\frac{1}{q}}\\&
=\left(\int_{|x|}^{\infty}\int_{\mathfrak{S}}r^{Q-1}r^{(\alpha+\lambda)q}drd\sigma(\omega)\right)^{\frac{1}{q}}\\&
=\left(|\mathfrak{S}|\int_{|x|}^{\infty}r^{Q-1+(\alpha+\lambda)q}dr\right)^{\frac{1}{q}}\\&
=\left(-\frac{|\mathfrak{S}|}{Q+(\alpha+\lambda)q}|x|^{Q+(\alpha+\lambda)q}\right)^{\frac{1}{q}}\\&
=\left(\frac{|\mathfrak{S}|}{|Q+(\alpha+\lambda)q|}\right)^{\frac{1}{q}} |x|^{\frac{Q+(\alpha+\lambda)q}{q}}.
\end{split}
\end{equation*}
Since $\beta<-\frac{Q}{p'}$, we have
$$-\beta p(1-p')+Q>-\beta p(1-p')-\beta p'=0 .$$
So, $-\beta p(1-p')+Q>0$. Then, let us consider
\begin{equation}
\begin{split}
\left(\int_{ B(0,|x|)}U^{1-p'}(y)dy\right)^{\frac{1}{p'}}&=\left(\int_{ B(0,|x|)}|y|^{-\beta p(1-p')}dy\right)^{\frac{1}{p'}}\\&
=\left(\int^{|x|}_{0}\int_{\mathfrak{S}}r^{-\beta p(1-p')}r^{Q-1}drd\sigma(\omega)\right)^{\frac{1}{p'}}
\\&
=\left(|\mathfrak{S}|\int^{|x|}_{0}r^{-\beta p(1-p')+Q-1}dr\right)^{\frac{1}{p'}}\\&
=\left(\frac{|\mathfrak{S}|}{-\beta p(1-p')+Q}|x|^{-\beta p(1-p')+Q}\right)^{\frac{1}{p'}}\\&
= \left(\frac{|\mathfrak{S}|}{Q-\beta p(1-p')}\right)^{\frac{1}{p'}} |x|^{\frac{-\beta p(1- p')+Q}{p'}}.
\end{split}
\end{equation}
Therefore, we have
\begin{equation}\label{a1}
\begin{split}
A_{1}&=\inf_{x\neq a}\left(\int_{\mathbb{G}\setminus B(0,|x|)}W(y)dx\right)^{\frac{1}{q}}\left(\int_{ B(0,|x|)}U^{1-p'}(y)dy\right)^{\frac{1}{p'}}\\&
= \left(\frac{|\mathfrak{S}|}{|Q+(\alpha+\lambda)q|}\right)^{\frac{1}{q}}\left(\frac{|\mathfrak{S}|}{Q-\beta p(1-p')}\right)^{\frac{1}{p'}} \inf_{x\neq a} |x|^{\frac{(\alpha+\lambda)q+Q}{q}+\frac{-\beta p(1- p')+Q}{p'}}\\&
= \left(\frac{|\mathfrak{S}|}{|Q+(\alpha+\lambda)q|}\right)^{\frac{1}{q}}\left(\frac{|\mathfrak{S}|}{Q-\beta p(1-p')}\right)^{\frac{1}{p'}}\inf_{x\neq a}|x|^{Q\left(\frac{1}{q}+\frac{1}{p'}+\frac{\alpha+\beta+\lambda}{Q}\right)}\\&
=\left(\frac{|\mathfrak{S}|}{|Q+(\alpha+\lambda)q|}\right)^{\frac{1}{q}}\left(\frac{|\mathfrak{S}|}{Q-\beta p(1-p')}\right)^{\frac{1}{p'}}\inf_{x\neq a}|x|^{Q\left(2-\frac{1}{q'}-\frac{1}{p}+\frac{\alpha+\beta+\lambda}{Q}\right)}\\&
=\left(\frac{|\mathfrak{S}|}{|Q+(\alpha+\lambda)q|}\right)^{\frac{1}{q}}\left(\frac{|\mathfrak{S}|}{Q-\beta p(1-p')}\right)^{\frac{1}{p'}}>0.
\end{split}
\end{equation}
Then by using \eqref{5.2}, we obtain
\begin{equation}
I_{1}\leq 2^{-\lambda q}\int_{\mathbb{G}}|x|^{(\alpha+\lambda)q}\left(\int_{B\left(0,\frac{|x|}{2}\right)}u(y)dy\right)^{q}dx
\leq 2^{-\lambda q}C^{q}_{1}\||y|^{-\beta}u\|^{q}_{L^{p}(\mathbb{G})},
\end{equation}
so that
\begin{equation}\label{I1}
I^{\frac{1}{q}}_{1}\geq  2^{-\lambda }C_{1}\||y|^{-\beta}u\|_{L^{p}(\mathbb{G})}=2^{-\lambda }C_{1}\|h\|_{L^{p}(\mathbb{G})}.
\end{equation}

\textbf{Step 2.} As in the previous case $I_{1}$, now we consider $I_{2}$. From $2|x|\leq |y|$,  we calculate
$$|y|=|y^{-1}|=|y^{-1} x x^{-1}|\stackrel{\eqref{tri}}\leq |y^{-1} x|+|x|\leq |y^{-1} x|+\frac{|y|}{2},$$
that is,
$$\frac{|y|}{2}\leq |y^{-1} x|.$$
Then, if condition \eqref{5.4.1} with $W(x)=|x|^{\alpha q}$ and $U(y)=|y|^{-(\beta+\lambda)p}$ is satisfied, then we have
\begin{multline*}
I_{2}=\int_{\mathbb{G}}\left(\int_{\mathbb{G}\setminus B(0,2|x|)}|x|^{\alpha}|y^{-1} x|^{\lambda}u(y)dy\right)^{q}dx\\
\leq 2^{-\lambda q}\int_{\mathbb{G}}|x|^{\alpha q}\left(\int_{\mathbb{G}\setminus B(0,2|x|)}u(y)|y|^{\lambda}dy\right)^{q}dx
\leq 2^{-\lambda q}\||y|^{-\beta}u\|^{q}_{L^{p}(\mathbb{G})}.
\end{multline*}
Now let us check condition \eqref{5.4.1}. We have
\begin{equation*}
\begin{split}
\left(\int_{ B(0,|x|)}W(y)dy\right)^{\frac{1}{q}}&=\left(\int_{ B(0,|x|)}|y|^{\alpha q}dy\right)^{\frac{1}{q}}\\&
=\left(\int_{0}^{|x|}\int_{\mathfrak{S}}r^{\alpha q}r^{Q-1}drd\sigma(\omega)\right)^{\frac{1}{q}}\\&
= \left(\frac{|\mathfrak{S}|}{Q+\alpha q}\right)^{\frac{1}{q}} |x|^{\frac{Q+\alpha q}{q}},
\end{split}
\end{equation*}
where $Q+\alpha q>0$. From $\alpha<-\frac{Q}{q}$, we have
\begin{equation*}
\begin{split}
\frac{1}{q'}+\frac{1}{p}=\frac{\alpha+\beta+\lambda}{Q}+2<-\frac{1}{q}+\frac{\beta+\lambda}{Q}+2=\frac{\beta+\lambda}{Q}+1+\frac{1}{q'},
\end{split}
\end{equation*}
then
\begin{equation}\label{betalambda}
(\beta+\lambda)p'+Q<0.
\end{equation}
By using this fact, we have
\begin{equation*}
\begin{split}
\left(\int_{\mathbb{G}\setminus B(0,|x|)}U^{1-p'}(y)dy\right)^{\frac{1}{p'}}&=\left(\int_{\mathbb{G}\setminus B(0,|x|)}|y|^{-(\beta+\lambda)(1-p')p}dy\right)^{\frac{1}{p'}}\\&
=\left(\int_{|x|}^{\infty}\int_{\mathfrak{S}}r^{Q-1}r^{-(\beta+\lambda)(1-p')p}drd\sigma(\omega)\right)^{\frac{1}{p'}}\\&
=\left(|\mathfrak{S}|\int_{|x|}^{\infty}r^{-(\beta+\lambda)(1-p')p+Q-1}dr\right)^{\frac{1}{p'}}\\&
\stackrel{\eqref{betalambda}}=\left(-\frac{|\mathfrak{S}|}{Q-(\beta+\lambda)(1-p')p}|x|^{Q-(\beta+\lambda)(1-p')p}\right)^{\frac{1}{p'}}\\&
=\left(-\frac{|\mathfrak{S}|}{Q+(\beta+\lambda)p'}|x|^{Q+(\beta+\lambda)p'}\right)^{\frac{1}{p'}}\\&
= \left(\frac{|\mathfrak{S}|}{|Q+(\beta+\lambda)p'|}\right)^{\frac{1}{p'}} |x|^{\frac{Q+(\beta+\lambda)p'}{p'}}.
\end{split}
\end{equation*}
Combining these facts we have
\begin{equation}\label{a2}
\begin{split}
A_{2}&=\inf_{x\neq a}\left(\int_{ B(0,|x|)}W(y)dx\right)^{\frac{1}{q}}\left(\int_{\mathbb{G}\setminus B(0,|x|)}U^{1-p'}(y)dx\right)^{\frac{1}{p'}}\\&
= \left(\frac{|\mathfrak{S}|}{Q+\alpha q}\right)^{\frac{1}{q}}\left(\frac{|\mathfrak{S}|}{|Q+(\beta+\lambda)p'|}\right)^{\frac{1}{p'}}\inf_{x\neq a} |x|^{\frac{Q+\alpha q}{q}+\frac{Q+(\beta+\lambda)p'}{p'}}\\&
=\left(\frac{|\mathfrak{S}|}{Q+\alpha q}\right)^{\frac{1}{q}}\left(\frac{|\mathfrak{S}|}{|Q+(\beta+\lambda)p'|}\right)^{\frac{1}{p'}}\inf_{x\neq a} |x|^{\frac{Q}{q}+\alpha+\frac{Q}{p'}+\beta+\lambda}\\&
=\left(\frac{|\mathfrak{S}|}{Q+\alpha q}\right)^{\frac{1}{q}}\left(\frac{|\mathfrak{S}|}{|Q+(\beta+\lambda)p'|}\right)^{\frac{1}{p'}}\inf_{x\neq a} |x|^{Q\left(\frac{1}{q}+\frac{1}{p'}+\frac{\alpha+\beta+\lambda}{Q}\right)}\\&
=\left(\frac{|\mathfrak{S}|}{Q+\alpha q}\right)^{\frac{1}{q}}\left(\frac{|\mathfrak{S}|}{|Q+(\beta+\lambda)p'|}\right)^{\frac{1}{p'}}\inf_{x\neq a} |x|^{Q\left(2-\frac{1}{q'}-\frac{1}{p}+\frac{\alpha+\beta+\lambda}{Q}\right)}\\&
=\left(\frac{|\mathfrak{S}|}{Q+\alpha q}\right)^{\frac{1}{q}}\left(\frac{|\mathfrak{S}|}{|Q+(\beta+\lambda)p'|}\right)^{\frac{1}{p'}}>0.
\end{split}
\end{equation}
Therefore, we have
\begin{equation*}
I_{2}=\int_{\mathbb{G}}\left(\int_{\mathbb{G}\setminus B(0,2|x|)}|x|^{\alpha}u(y)|y^{-1} x|^{\lambda}dy\right)^{q}dx\leq 2^{-\lambda q}C^{q}_{2}\||y|^{-\beta}u\|^{q}_{L^{p}(\mathbb{G})}.
\end{equation*}
Then, we have
\begin{equation}\label{I3}
I^{\frac{1}{q}}_{2}\geq 2^{-\lambda }C_{2}\||y|^{-\beta}u\|_{L^{p}(\mathbb{G})}=2^{-\lambda }C_{2}\|h\|_{L^{p}(\mathbb{G})}.
\end{equation}
Finally, by using \eqref{I1} and \eqref{I3} in \eqref{i13}, we obtain
\begin{equation}
\begin{split}
\left(\int_{\mathbb{G}}|x|^{\alpha q}\left(\int_{\mathbb{G}}|y^{-1} x|^{\lambda}u(y)dy\right)^{q}dx\right)^{\frac{1}{q}}&
\geq\frac{I^{\frac{1}{q}}_{1}}{2}+\frac{I^{\frac{1}{q}}_{2}}{2}\\&
\geq \frac{2^{-\lambda}(C_{1}+C_{2})}{2}\||y|^{-\beta}u\|_{L^{p}(\mathbb{G})}\\&
= \frac{2^{-\lambda }(C_{1}+C_{2})}{2}\||y|^{-\beta}u\|_{L^{p}(\mathbb{G})}\\&
=C_{3}\||y|^{-\beta}u\|_{L^{p}(\mathbb{G})},
\end{split}
\end{equation}
where $C_{3}=\frac{2^{-\lambda}(C_{1}+C_{2})}{2}>0$.

Theorem \ref{stein-weiss3} is proved.
\end{proof}
Let us give improved reverse Stein-Weiss inequality.
\begin{thm}\label{stein-weiss4}
Let $\mathbb{G}$ be a homogeneous group of homogeneous dimension $Q\geq1$ and let $|\cdot|$ be an arbitrary homogeneous quasi-norm on $\mathbb{G}$.
Let $\lambda>0$, $p,q'\in(0,1)$ and $\frac{1}{q'}+\frac{1}{p}=\frac{\alpha+\beta+\lambda}{Q}+2$, where $\frac{1}{p}+\frac{1}{p'}=1$  , $\frac{1}{q}+\frac{1}{q'}=1$. Then for all non-negative functions $f\in L^{q'}(\mathbb{G})$ and $h\in L^{p}(\mathbb{G})$, inequality \eqref{stein-weiss} holds, that is,
\begin{equation*}
\int_{\mathbb{G}}\int_{\mathbb{G}}|x|^{\alpha}|y^{-1}x|^{\lambda}f(x)h(y)|y|^{\beta}dxdy\geq C\|f\|_{L^{q'}(\mathbb{G})}\|h\|_{L^{p}(\mathbb{G})},
\end{equation*}
if  one of the following conditions is satisfied:
\begin{itemize}
	\item[(a)] $0\leq\alpha<-\frac{Q}{q}$. 

	\item[(b)] $0\leq\beta<-\frac{Q}{p'}$.

\end{itemize}
\end{thm}
\begin{proof}
Let us prove (a). By using some notations from proof of  Theorem \ref{stein-weiss3} and \eqref{i13}, we obtain
\begin{equation}
\left(\int_{\mathbb{G}}|x|^{\alpha q}\left(\int_{\mathbb{G}}|y^{-1} x|^{\lambda}u(y)dy\right)^{q}dx\right)^{\frac{1}{q}}\geq I^{\frac{1}{q}}_{2},
\end{equation}
and by using Step 2 in the proof of Theorem \ref{stein-weiss3}  and from \eqref{I3}, we have $I_{2}^{\frac{1}{q}}\geq C\||y|^{-\beta}u\|_{L^{p}(\mathbb{G})},$ then we get
\begin{equation}
\left(\int_{\mathbb{G}}|x|^{\alpha q}\left(\int_{\mathbb{G}}|y^{-1} x|^{\lambda}u(y)dy\right)^{q}dx\right)^{\frac{1}{q}}\geq I^{\frac{1}{q}}_{2}\geq C\||y|^{-\beta}u\|_{L^{p}(\mathbb{G})}.
\end{equation}

Let us prove (b). By \eqref{i13}, we obtain
\begin{equation}
\left(\int_{\mathbb{G}}|x|^{\alpha q}\left(\int_{\mathbb{G}}|y^{-1} x|^{\lambda}u(y)dy\right)^{q}dx\right)^{\frac{1}{q}}\geq I^{\frac{1}{q}}_{1},
\end{equation}
and by using Step 1 in the proof of Theorem \ref{stein-weiss3} and from \eqref{i1}, we have $I_{1}^{\frac{1}{q}}\geq C\||y|^{-\beta}u\|_{L^{p}(\mathbb{G})},$ then we get
\begin{equation}
\left(\int_{\mathbb{G}}|x|^{\alpha q}\left(\int_{\mathbb{G}}|y^{-1} x|^{\lambda}u(y)dy\right)^{q}dx\right)^{\frac{1}{q}}\geq I^{\frac{1}{q}}_{1}\geq C\||y|^{-\beta}u\|_{L^{p}(\mathbb{G})}.
\end{equation}

\end{proof}

Let us give reverse Hardy, $L^{p}$-Sobolev and $L^{p}$-Caffarelli-Kohn-Nirenberg inequalities on $\G$.
Assume now that  $f$ is a radially decreasing function, i.e., $\R f:=\frac{d}{d|x|}f<0$. Let us give the reverse Hardy inequality on homogeneous Lie groups, the reverse to Theorem \ref{dirHardy}.
\begin{thm}[Reverse Hardy inequality]
Let $\G$ be a homogeneous Lie group with homogeneous dimension $Q\geq1$. Assume that $p\in(0,1)$. Then for any non-negative, real-valued and radially decreasing function $f\in C^{\infty}_{0}(\G\setminus\{0\})$, we have
\begin{equation}\label{revhar}
\left\|\frac{f}{|x|}\right\|_{L^{p}(\G)}\geq \frac{p}{Q-p}\|\R f\|_{L^{p}(\G)}.
\end{equation}
\end{thm}
\begin{proof}
Let us denote $\R_{1}=-\R$, so that we have $\R_{1}f>0$. By using polar decomposition \eqref{EQ:polar}, integration by parts and reverse H\"{o}lder's inequality, we obtain
\begin{equation}
\begin{split}
\int_{\G}\frac{f^{p}(x)}{|x|^{p}}dx&=\int_{0}^{\infty}\int_{\S}\frac{f^{p}(ry)}{r^{p}}r^{Q-1}drd\sigma(y)\\&
=-\frac{p}{Q-p}\int_{\G}\frac{f^{p-1}(x)}{|x|^{p-1}}\R f(x)dx\\&
=\frac{p}{Q-p}\int_{\G}\frac{f^{p-1}(x)}{|x|^{p-1}}\R_{1} f(x)dx\\&
\geq\frac{p}{Q-p}\left\|\frac{f}{|x|}\right\|^{p-1}_{L^{p}(\G)}\|\R_{1}f\|_{L^{p}(\G)}.
\end{split}
\end{equation}
This gives
\begin{equation}
\left\|\frac{f}{|x|}\right\|_{L^{p}(\G)}\geq\frac{p}{Q-p}\|\R_{1}f\|_{L^{p}(\G)},
\end{equation}
implying \eqref{revhar}.
\end{proof}
Let us define by $\mathbb{E}=|x|\R$ the Euler operator.   Then we have the reverse $L^{p}$-Sobolev inequality, the reverse to Theorem \ref{dirSob}.
\begin{thm}[Reverse $L^{p}$-Sobolev inequality]
Let $\G$ be a homogeneous Lie group with homogeneous dimension $Q\geq1$. Assume that $p\in(0,1)$. Then for any non-negative, real-valued and radially decreasing function $f\in C^{\infty}_{0}(\G\setminus\{0\})$,  we have
\begin{equation}\label{sobrev}
\left\|f\right\|_{L^{p}(\G)}\geq \frac{p}{Q}\|\mathbb{E}f\|_{L^{p}(\G)}.
\end{equation}
\end{thm}
\begin{proof}
Let us denote $\mathbb{E}_{1}=|x|\R_{1}$, so that $\mathbb{E}_{1}f>0.$ By using polar decomposition \eqref{EQ:polar}, integration by parts and reverse H\"{o}lder's inequality, we obtain
\begin{equation}
\begin{split}
\int_{\G}f^{p}(x)dx&=\int_{0}^{\infty}\int_{\S}f^{p}(ry)r^{Q-1}drd\sigma(y)\\&
=-\frac{p}{Q}\int_{\G}f^{p-1}(x)|x|\R f(x)dx\\&
=\frac{p}{Q}\int_{\G}f^{p-1}(x)|x|\R_{1} f(x)dx\\&
=\frac{p}{Q}\int_{\G}f^{p-1}(x)\mathbb{E}_{1} f(x)dx\\&
\geq\frac{p}{Q}\left\|f\right\|^{p-1}_{L^{p}(\G)}\|\mathbb{E}_{1} f\|_{L^{p}(\G)}.
\end{split}
\end{equation}
This gives
\begin{equation}
\left\|f\right\|_{L^{p}(\G)}\geq\frac{p}{Q}\|\mathbb{E}_{1}f\|_{L^{p}(\G)},
\end{equation}
implying \eqref{sobrev}.
\end{proof}
Let us give the reverse $L^{p}$-Caffarelli-Kohn-Nirenberg inequality on $\G$, that is, the reverse to Theorem \ref{LpCKN}.
\begin{thm}[Reverse $L^{p}$-Caffarelli-Kohn-Nirenberg inequality]
Let $\G$ be a homogeneous Lie group with homogeneous dimension $Q\geq1$. Assume that $p\in(0,1)$. Then for any nonnegative, real-valued and radially decreasing function $f\in C^{\infty}_{0}(\G\setminus\{0\})$, we have
\begin{equation}\label{CKN}
\left\|\frac{f}{|x|^{\frac{\gamma}{p}}}\right\|^{p}_{L^{p}(\G)}\geq \frac{p}{Q-\gamma}\left\|\frac{\R f}{|x|^{\alpha}}\right\|_{L^{p}(\G)}\left\|\frac{f}{|x|^{\frac{\beta}{p-1}}}\right\|^{p-1}_{L^{p}(\G)},
\end{equation}
for all $\alpha,\beta\in \mathbb{R}$ and $\gamma=\alpha+\beta+1$, such that $Q> \gamma$.
\end{thm}
\begin{proof}
By using polar decomposition \eqref{EQ:polar}, integration by parts and reverse H\"{o}lder's inequality, we obtain
\begin{equation}
\begin{split}
\int_{\G}\frac{f^{p}(x)}{|x|^{\gamma}}dx&=\int_{0}^{\infty}\int_{\S}\frac{f^{p}(ry)}{r^{\gamma}}r^{Q-1}drd\sigma(y)\\&
=-\frac{p}{Q-\gamma}\int_{\G}\frac{f^{p-1}(x)}{|x|^{\gamma-1}}\R f(x)dx\\&
=\frac{p}{Q-\gamma}\int_{\G}\frac{f^{p-1}(x)}{|x|^{\alpha+\beta}}\R_{1} f(x)dx\\&
=\frac{p}{Q-\gamma}\int_{\G}\frac{f^{p-1}(x)}{|x|^{\beta}}\frac{\R_{1} f(x)}{|x|^{\alpha}}dx\\&
\geq\frac{p}{Q-\gamma}\left\|\frac{f}{|x|^{\frac{\beta p'}{p}}}\right\|^{p-1}_{L^{p}(\G)}\left\|\frac{\R_{1}f}{|x|^{\alpha}}\right\|_{L^{p}(\G)}\\&
=\frac{p}{Q-\gamma}\left\|\frac{f}{|x|^{\frac{\beta \frac{p}{p-1}}{p}}}\right\|^{p-1}_{L^{p}(\G)}\left\|\frac{\R_{1}f}{|x|^{\alpha}}\right\|_{L^{p}(\G)}\\&
=\frac{p}{Q-\gamma}\left\|\frac{f}{|x|^{\frac{\beta}{p-1}}}\right\|^{p-1}_{L^{p}(\G)}\left\|\frac{\R_{1}f}{|x|^{\alpha}}\right\|_{L^{p}(\G)}.
\end{split}
\end{equation}
This gives
\begin{equation}
\left\|\frac{f}{|x|^{\frac{\gamma}{p}}}\right\|^{p}_{L^{p}(\G)}\geq \frac{p}{Q-\gamma}\left\|\frac{\R_{1}f}{|x|^{\alpha}}\right\|_{L^{p}(\G)}\left\|\frac{f}{|x|^{\frac{\beta}{p-1}}}\right\|^{p-1}_{L^{p}(\G)},
\end{equation}
which implies \eqref{CKN}.
\end{proof}
\begin{rem} In \eqref{CKN}, if we take $\gamma=p$ and $\alpha=0$, then we have the reverse Hardy inequality. Also, if we take $\gamma=0$ and $\beta=0$, then we have the reverse $L^{p}$-Sobolev inequality.
\end{rem}


\begin{thebibliography}{H8}
\bibitem{AF}
R. A. Adams and J. J. F. Fournier. Sobolev spaces. Vol. 140. Elsevier, 2003.
\bibitem{CKN}
L. A.~Caffarelli, R.~Kohn and L.~Nirenberg.
\newblock First order interpolation inequalities with weights.
\newblock  {\em Composito Math.}, 53(3):259--275, 1984.

\bibitem{CCR}
P. Ciatti, M. G. Cowling, and F. Ricci. Hardy and uncertainty inequalities on stratified Lie groups. {\em Adv. Math.}, 277:365--387, 2015.
\bibitem{CLT}
L.~Chen, G.~Lu and C.~Tao.  Reverse Stein--Weiss Inequalities on the Upper Half Space and the Existence of Their Extremals. {\em Advanced Nonlinear Studies}, 19(3):475-494, 2019.
\bibitem{CLT1}
L.~Chen, Z.~ Liu, G.~Lu, and C.~Tao.  Reverse Stein-Weiss inequalities and existence of their extremal functions. {\em Transactions of the American Mathematical Society}, 370(12):8429--8450, 2018.

\bibitem{DA} L. D'Ambrosio. Hardy-type inequalities related to degenerate elliptic differential operators. {\em Ann. Sc. Norm. Super. Pisa Cl. Sci. (5)}, 4(3): 45186, 2005.

 \bibitem{DGP} D. Danielli, N. Garofalo, and N. C. Phuc. Hardy-Sobolev type inequalities with sharp constants in Carnot-Carath\'{e}odory spaces. {\em Potential Anal.}, 
34 (3): 22342, 2011. 


\bibitem{JD}
J.~Dou.
\newblock Weighted Hardy-Littlewood-Sobolev inequalities on the upper half space.
\newblock {\em Commun. Contemp. Math.}, 18, 1550067, 2016.


\bibitem{JZ}
J. Dou and M. Zhu. Reversed Hardy-Littewood-Sobolev inequality. {\em Int. Math. Res.
Not. IMRN}, 19:9696--9726, 2015.
\bibitem{CF}
R. R. Coifman and C. Fefferman.
\newblock Weighted Norm Inequalities for Maximal Functions and Singular Integrals.
\newblock {\em Studia Mathematica}, 51: 241-250, 1974.
\bibitem{CDDFF}
J. A. Carrillo, M. G. Delgadino, J. Dolbeault, R. L. Frank and F. Hoffmann.  Reverse Hardy-Littlewood-Sobolev inequalities. {\em Journal de Math\'{e}matiques Pures et Appliqu\'{e}es}, 132:133--165, 2019.
\bibitem{FM}
 C. Fefferman and B. Muckenhoupt.
 \newblock Two Nonequivalent Conditions for Weight Functions.
\newblock {\em Proceeding of the American Mathematical Society}, 45: 99-104, 1974.
	\bibitem{FR}
V. Fischer,	M.~Ruzhansky.
	\newblock Quantization on nilpotent Lie groups.
	\newblock Progress in Mathematics, Vol. 314, Birkhauser, 2016. (open access book)
	
	\bibitem{FS74}
G.~B. Folland and E.~M. Stein.
\newblock Estimates for the $\overline{\partial_{b}}$ complex and
analysis on the Heisenberg group.
\newblock {\em Comm. Pure Appl. Math.}, 27:429--522, 1974.
	
\bibitem{FS1}
G. B.~Folland and E. M.~Stein.
\newblock Hardy spaces on homogeneous groups.
\newblock Mathematical Notes, Vol. 28, Princeton University Press, Princeton, N.J.; University of
Tokyo Press, Tokyo, 1982.

\bibitem{FL12}
R.~L.~Frank and E.~H~Lieb.
\newblock Sharp constants in several inequalities on
the Heisenberg group.
\newblock {\em Ann. of Math.}, 176:349--381, 2012.
\bibitem{GL}
N. Garofalo and E. Lanconelli. Frequency functions on the Heisenberg group,
the uncertainty principle and unique continuation. {\em Ann. Inst. Fourier (Grenoble)},
40(2):313--356, 1990.
 \bibitem{GMS}
 V.~Guliyev, R.~Mustafayev and A.~Serbetci.
\newblock Stein-Weiss inequalities for the fractional integral operators in Carnot groups and applications.
\newblock{\em Complex Variables and Elliptic Equations}, 55:8--10, 847--863, 2010.
\bibitem{HHAT}
M. Hoffmann-Ostenhof, T. Hoffmann-Ostenhof, A. Laptev and J. Tidblom. Manyparticle Hardy inequalities. \newblock{\em J. Lond. Math. Soc. (2)}, 77(1):99--114, 2008.
\bibitem{HL}
T. Hoffmann-Ostenhof and A. Laptev. Hardy inequalities with homogeneous weights.
{\em J. Funct. Anal.,} 268(11):3278--3289, 2015.
\bibitem{KRS}
 A. Kassymov, M. Ruzhansky and D. Suragan. Hardy-Littlewood-Sobolev and Stein-Weiss inequalities on homogeneous Lie groups. {\em Integral Transform. Spec. Funct.}, 30(8):643--655, 2019.
\bibitem{KRS2}
 A. Kassymov, M. Ruzhansky and D. Suragan. Reverse integral Hardy inequality on metric measure spaces. {\em arXiv:2942713}, 2019.

\bibitem{HLZ}
X.~Han, G.~Lu and  J.~Zhu.
\newblock Hardy-Littlewood-Sobolev and Stein-Weiss inequalities and integral systems on the Heisenberg group.
\newblock {\em Nonlinear Analysis}, 75, 4296--4314, 2012.

\bibitem{HL28}
G.~H.~Hardy and J.~E.~Littlewood.
\newblock Some properties of fractional integrals.
\newblock {\em I. Math.Z.}, 27:565--606, 1928.

\bibitem{Lie83}
E.~H.~Lieb.
\newblock Sharp constants in the Hardy-Littlewood-Sobolev and related
inequalities.
\newblock {\em Ann. of Math.}, 118:349--374, 1983.
\bibitem{MW}
B. Muckenhoupt and R. L. Wheeden.
\newblock Weighted Norm Inequality for Fractional Integrals.
\newblock {\em Transactions of the American Mathematical Society}, 192: 261-274, 1974.
\bibitem{NN}
Q. A. Ng\^{o} and V. Nguyen. Sharp reversed Hardy-Littlewood-Sobolev inequality on
$R^n$. {\em Israel J. Math.}, 220:189--223, 2017.
\bibitem{OS}
T. Ozawa and H. Sasaki: Inequalities associated with dilations.
{\em Commun. Contemp. Math.}, 11(2):265--277, 2009.
\bibitem{ORS}
T.~Ozawa, M.~Ruzhansky, D.~Suragan. $L^{p}$-Caffarelli-Kohn-Nirenberg type inequalities on homogeneous groups. {\em Quart. J. Math.,} 70: 305--318, 2019.
\bibitem{Per}
C. Perez.
 \newblock Two Weighted Norm Inequalities for Riesz Potentials and Uniform $L^{p}$-Weighted Sobolev Inequalities.
\newblock {\em Indiana University Mathematics Journal}, 39: 31--44, 1990.

\bibitem{RSS1}
M.~Ruzhansky, B.~Sabitbek and D.~Suragan. Hardy and Rellich inequalities for anisotropic p-sub-Laplacians, and horizontal Hardy inequalities for multiple singularities and multi-particles on stratified groups. {\em Banach Journal of Mathematical Analysis}, to appear.
\bibitem{RSS2}
M.~Ruzhansky, B.~Sabitbek and D. Suragan. Weighted $L^{p}$-Hardy and $L^{p}$-Rellich inequalities with boundary terms on stratified Lie groups. {\em Rev. Mat. Complutense}, 32:19-35, 2019.
\bibitem{RSS3}
M.~Ruzhansky, B.~Sabitbek and D. Suragan. Subelliptic geometric Hardy type inequalities on half-spaces and convex domains.
{\em arXiv:1806.06226}, 2018.
\bibitem{RSS4}
M.~Ruzhansky, B.~Sabitbek and D. Suragan. Weighted anisotropic Hardy and Rellich type inequalities for general vector fields, {\em NoDEA Nonlinear Differential Equations Appl.}, 26(2), 26:13, 2019.
\bibitem{RS_book}
M.~Ruzhansky and D.~Suragan.
\newblock Hardy inequalities on homogeneous groups.
\newblock \textit{Progress in Math.} Vol. 327, Birkh\"auser, 588 pp, 2019 (open access book).

\bibitem{RSAMS}
M.~Ruzhansky and D.~Suragan. Layer potentials, Green formulae, Kac problem, and refined Hardy inequality on homogeneous Carnot groups. {\em Adv. Math.}, 308:483-528, 2017.
\bibitem{RSAM}
M.~Ruzhansky and D.~Suragan.
 \newblock Hardy and Rellich inequalities, identities, and sharp remainders on homogeneous groups.
 \newblock{\em Adv. Math.}, 317, 799--822, 2017.
\bibitem{RSVF}
M.~Ruzhansky and D.~Suragan. Local Hardy and Rellich inequalities for sums of squares of vector fields. {\em Adv. Diff. Equations}, 22:505--540, 2017
\bibitem{RSY1}
M.~Ruzhansky, D.~Suragan and N.~Yessirkegenov.
\newblock Extended Caffarelli-Kohn-Nirenberg inequalities, and remainders, stability, and superweights for $L^p$-weighted
Hardy inequalities.
\newblock {\em Trans. Amer. Math. Soc. Ser. B}, 5:32--62, 2018.
\bibitem{RSY2}
M.~Ruzhansky, D.~Suragan and N.~Yessirkegenov. Caffarelli-Kohn-Nirenberg and Sobolev type inequalities on stratified Lie groups. {\em NoDEA}, 24(5), 24:56, 2017.
\bibitem{RSY3}
M.~Ruzhansky, D.~Suragan and N.~Yessirkegenov. Sobolev type inequalities, Euler-Hilbert-Sobolev and Sobolev-Lorentz-Zygmund spaces on homogeneous groups. {\em Integral Equations Operator Theory}, 90, no. 1, 90:10, 2018.

\bibitem{RY1}
M.~Ruzhansky and N.~Yessirkegenov. Factorizations and Hardy-Rellich inequalities on stratified groups. {\em J. Spectr. Theory}, to appear.
\bibitem{RY}
M.~Ruzhansky and N.~Yessirkegenov.
\newblock Hypoelliptic functional inequalities.
\newblock {\em arXiv:1805.01064v1}, 2018.

\bibitem{Sob38}
S.~L.~Sobolev.
\newblock On a theorem of functional analysis.
\newblock {\em Mat. Sb. (N.S.)}, 4:471--479, 1938,
\newblock English transl. in Amer. Math. Soc. Transl. Ser. 2, 34:39--68, 1963.

\bibitem{steinbook}
E.~M.~Stein.
\newblock Singular Integrals and Differentiability Properties of Functions.
Princeton University Press, 1970.

\bibitem{StWe58}
E.~M.~Stein and G.~Weiss.
\newblock Fractional integrals on $n$-dimensional Euclidean Space.
\newblock {\em J. Math. Mech.}, 7(4):503--514, 1958.






\end{thebibliography}
\end{document}